\documentclass[12pt]{article}

\usepackage{amsmath,amsthm}
\usepackage{graphicx}
\usepackage{color}
\usepackage{appendix}
\usepackage{amssymb}
\newtheorem{theorem}{Theorem}[section]
\newtheorem{proposition}{Proposition}[section]

\newtheorem{remark}{Remark}[section]

\newtheorem{lemma}{Lemma}[section]

\newcommand{\red}[1]{\textcolor{red}{#1}}

\allowdisplaybreaks

\begin{document}

\title{\bf Mean Field Approximations to a Queueing System with Threshold-Based Workload Control Scheme}
\author{Qihui Bu$^{1}$, Liwei Liu$^{*1}$ and Yiqiang Q. Zhao$^{2}$\\
{\em\small 1: School of Science, Nanjing University of Science and Technology}\\
{\em\small Nanjing 210094, Jiangsu, China}\\
{\em\small 2: School of Mathematics and Statistics, Carleton University}\\
{\em\small Ottawa, Ontario Canada  K1S 5B6}\\
{\em\small\em $^{*}$Corresponding Author: lwliu@mail.njust.edu.cn}}
\date{}
\maketitle

\renewcommand{\baselinestretch}{1.3}
\large\normalsize
\begin{abstract}
  In this paper, motivated by considerations of server utilization and energy consumptions in cloud computing, we investigate a homogeneous queueing system with a threshold-based workload control scheme. In this system, a virtual machine will be turned off when there are no tasks in its buffer upon the completion of a service by the machine,  and turned on when the number of tasks in its buffer reaches a pre-set threshold value. Due to complexity of this system, we propose approximations to system performance measures by mean field limits. An iterative algorithm is suggested for the solution to the mean field limit equations. In addition, numerical and simulation results are presented to justify the proposed approximation method and to provide a numerical analysis on the impact of the system performances by system parameters.
\par
{\bf Keywords:} Mean field limit, threshold service queue, join-the-shortest queue, Markov process.
\end{abstract}
\section{Introduction and model description}
The cloud computing paradigm has many applications in data centers, web server farms, next-generation wireless communication systems, such as cloud radio access networks, see for example Cai, Yu and Bu~\cite{CY} among others. As the development and maturity of technology, there are more and more people using cloud services, which makes cloud computing a significant part of our daily and professional lives. Very often, cloud providers supply a large number of resources, dynamically distributed, for supporting cloud computing services. The following is an abstract description of the structure of cloud computing: needs from users are presented to the system as tasks and each task is dispatched to a virtual machine, which is chosen from a group of available virtual machines by a load scheduling algorithm.

A central topic in the research on cloud computing is on load balancing algorithms, since a good load balancing algorithm often leads to good system performance.
There are many algorithms proposed in the literature, such as round robin, min-min and max-min algorithms, and active clustering load balancing, referring to Choudhary~\cite{MC}, Vakilinia, Ali and Qiu~\cite{VAQ}, and Dieker and Suk~\cite{DS} for more details.

 The power-of-$d$ scheme is considered in this paper, in which a task is dispatched to the shortest one among the $d$ randomly chosen queues. Vvedenskaya, Dobrushin and Karpelevich~\cite{VD} studied the queueing system, in which an arbitrary arrival customer joins the shortest one of the two randomly chosen queues and obtained that the tail of the equilibrium distribution decays doubly exponentially as the system size increases. In the last two decades, scholars, continuing previous studies, proposed multiple models with the power-of-$d$ scheme and discussed their performances in terms of the mean field limit.  Mitzenmacher~\cite{MM} considered a supermarket model by focusing on the impact on the expected sojourn time by the system parameters. Mazumdar and his coauthors studied a series of heterogeneous networks with mean field interactions (see, for example references Karthik, Mukhopadhyay and Mazumdar~\cite{KMM} and Mukhopadhyay~\cite{AM}). Li and Yang~\cite{Li} analyzed a simple framework for applying the mean field theory to dealing with a heterogeneous work stealing model. Dawson, Tang and Zhao~\cite{zhao} considered a general mean field interaction function in a queueing system, for which they obtained an explicit formula for the limit solution. Ying, Srikant and Kang~\cite{YSK} proposed a new load balancing algorithm, batch-filling, which can dramatically reduce the sample complexity compared to previously proposed algorithms.

 In all of the above systems, virtual machines for the most of time are powered on, which is very friendly to users, but also red raises issues such as less efficient utilization and more energy consumptions. Especially, as the demand for cloud service scales up, the utilization of virtual machines becomes more and more important and challenging to address. However, in the literature researchers paid more attention to performance measures, such as the waiting time for users, than that to the system utilization and the energy consumption problem. To fill this gap, in this paper we introduce a threshold-based workload control scheme, with which a virtual machine is forced to be turned off when the machine is idle until the number of tasks in its buffer reaches the pre-set threshold value. The concept of vacations is a similar idea to the threshold-based workload control scheme, which can reduce the energy consumptions and improve the efficiency of the system. Li \textit{el al.}~\cite{lql} studied a supermarket model with multiple vacations by the mean field limit, and provided approximation results through an iterative algorithm. An M/G/$1$ queueing model was constructed as an approximation limit method in LawanyaShri, Balusamy and Subha~\cite{MBS} to characterize the cloud computing and a cost function minimization problem was introduced to obtain the optimal threshold value. In addition, Mukherjee \textit{et al.}~\cite{tabs} proposed a joint auto-scaling and load balancing scheme, in which the global queue length information about the queue length is not required.

In this paper, we construct a queueing system with $N$ ($N\geq1$) identical and parallel virtual machines (servers), each having a buffer of infinite capacity. Tasks arrive to this system according to a Poisson process with rate $N\lambda$, where $\lambda>0$. Upon arrival, the task is dispatched to the virtual machine with the shortest queue among the $d$ randomly chosen machines. If there is a tie, virtual machines with the shortest queue are chosen randomly. The service times of each virtual machine are identical independent distributed with the exponential distribution of service rate $\mu$. When there is no task receiving service at the machine or in the buffer, the machine will be turned off and kept dormant until there are $M$ ($M\geq 2$) tasks in its buffer. This threshold-based workload control scheme makes sure that any idle virtual machines will be turned off to reduce the energy consumption and increase the performance. For this system, we apply the mean field limit theory to obtain an approximate model and calculate approximations to system stationary behaviours by an iterative algorithm.
The main contribution in this paper includes:
\begin{enumerate}
  \item A threshold-based workload control scheme, proposed for large-scaled parallel queueing networks to address issues of system utilization and energy consumptions and to improve the system performance;
  \item A mean field limiting process, constructed and justified to approximate the queueing system with a threshold-based workload control scheme;
  \item An iterative algorithm for (approximate) stationary performances of the queueing system, such as the mean queue length of any virtual machine and the mean sojourn time of an arbitrary arrival task;
  \item A numerical/simulation analysis for the impact of system parameters on the system performance, and for approximation errors. It is interesting to note that a new phenomenon/trend has been observed for the mean sojourn time, which can be longer for smaller values of the arrival rate, a property that does not hold for the system in which all virtue machines are always on.
\end{enumerate}

The rest of this paper is organized as follows: In Section~2, an infinite-dimensional system of mean field equations is proposed as the limit of the process for the original queueing system; the justification of using stationary probabilities of the limiting process as approximations to the corresponding stationary probabilities for the original queueing system is provided; and an iterative algorithm is provided for stationary behaviour of the queueing system; In Section~3, numerical/simulation analysis is performed to allow us to see performance trend or impact of system parameters on the system performance based on a large number of numerical computations/simulation runs; and finally in Section~4, concluding remarks are made.

\section{Analysis of the system}

In this section, we analyze the queueing system with threshold-based workload control scheme, proposed earlier. In Subsection~2.1, we introduce two sequences of fraction vectors, which are two Markov processes for characterizing this queueing system, and the corresponding limiting system. In Subsection~2.2, we construct mean field limit equations and prove that the original stochastic process for the queueing system converges to the solution to the limit equations. We provide an iterative algorithm for computing the fixed point of the limiting process and discuss the system stationary behaviour in Subsection~2.3.

\subsection{An infinite-dimensional Markov process}

We first define two sequences of random variables for the queueing system with threshold-based workload control scheme:
\begin{description}
 \item[$n_k^{(W)}(t)$:] the number of working virtual machines with at least $k$ $(k=1,2,3,\ldots)$ tasks (including the task being in execution) at time $t$;
  \item[$n_j^{(V)}(t)$:] the number of dormant virtual machines with at least $j$ $(j=0,1,2,\ldots, M-1)$ tasks at time $t$.
\end{description}
Clearly, we have the following equality and inequality relationships:
\[\begin{array}{c}
 \displaystyle n_1^{(W)}(t)+n_0^{(V)}(t)=N, \\[4mm]
 \displaystyle N\geq n_1^{(W)}(t)\geq n_2^{(W)}(t)\geq n_3^{(W)}(t)\geq \cdots\geq 0, \\[4mm]
 \displaystyle N\geq n_0^{(V)}(t)\geq n_1^{(V)}(t)\geq n_2^{(V)}(t)\geq \cdots \geq n_{M-1}^{(V)}(t)\geq 0.
\end{array}\]

Next, we introduce two important fractions,
\begin{eqnarray*}
 \displaystyle U^{(N)}_{k}(t) &=& \frac{n_k^{(W)}(t)}{N},~~~~ k=1,2,3,\ldots, \\[4mm]
\displaystyle  V^{(N)}_{j}(t) &=& \frac{n_j^{(V)}(t)}{N},~~~~j=0,1,2,\ldots,M-1,
\end{eqnarray*}
which are the fractions of working virtual machines with at least $k$ tasks, and dormant virtual machines with at least $j$ tasks at time $t\geq0$, respectively.

 Let
\begin{eqnarray*}
  \displaystyle \mathbf{U}^{(N)}(t) &=& \left(U^{(N)}_{1}(t), U^{(N)}_{2}(t), U^{(N)}_{3}(t),\ldots \right), \\[4mm]
 \displaystyle \mathbf{V}^{(N)}(t)&=& \left(V^{(N)}_{0}(t), V^{(N)}_{1}(t),V^{(N)}_{2}(t), \ldots, V^{(N)}_{M-1}(t)\right).
\end{eqnarray*}
Then, $\displaystyle \{(\mathbf{U}^{(N)}(t), \mathbf{V}^{(N)}(t)), t\geq0\}$ is an infinite-dimensional Markov process for characterizing the queueing system with threshold-based workload control scheme whose states are defined by
\begin{eqnarray*}
  \mathbf{u}^{(N)}&=&(u_1^N, u_2^N, u_3^N, \ldots),\\[4mm]
  \mathbf{v}^{(N)}&=&(v_0^N, v_1^N, \ldots, v_{M-1}^N),
\end{eqnarray*}
and the state space is defined by
\begin{eqnarray*}
  \displaystyle \Omega_{N} &=&  \{(\textbf{u}^{(N)}, \textbf{v}^{(N)}): 1\geq u_1^{(N)}\geq u_2^{(N)}\geq u_3^{(N)}\geq \cdots \geq0,\displaystyle 1\geq v_0^{(N)}\geq v_1^{(N)}\geq \cdots\geq v_{M-1}^{(N)} \geq 0, \\[4mm]
  && \displaystyle  Nu_k^{(N)}, Nv_j^{(N)}=0,1,2,\ldots,\textmd{ for } k\geq1, M-1\geq j\geq0\}.
\end{eqnarray*}

In the next subsection, a limiting system will be introduced, which is the limit of the Markov process $\displaystyle \{(\mathbf{U}^{(N)}(t), \mathbf{V}^{(N)}(t)), t\geq0\}$ when $N\longrightarrow \infty$. The state space $\Omega$ of this limiting process is described as follows:
\begin{eqnarray*}
  \mathbf{u}&=&(u_1, u_2, u_3, \ldots),\\[4mm]
            \mathbf{v}&=&(v_0,v_1,\ldots,v_{M-1}).
\end{eqnarray*}
Then, the state space is given by
\begin{equation*}
   \displaystyle \Omega=\{(\mathbf{u},\mathbf{v}): 1\geq u_1\geq u_2\geq \cdots\geq0, \displaystyle 1\geq v_0\geq v_1\geq \cdots\geq  v_{M-1} \geq0\}.
\end{equation*}
We adapt the following distance between $(\mathbf{u},\mathbf{v})$ and $(\mathbf{u}',\mathbf{v}')$ on $\Omega$:
\[\displaystyle \rho\left\{(\mathbf{u},\mathbf{v}),(\mathbf{u}',\mathbf{v}')\right\}
=\max\displaystyle\left\{\displaystyle\sup\limits_{k\in K}\mid\frac{u_k-u'_k}{k+1}\mid,\displaystyle \sup\limits_{j\in J}\mid\frac{v_j-v'_j}{j+1}\mid\right\},\]
where $K=\{1,2,\ldots\}$ and $J=\{0,1,2,\ldots,M-1\}$.

\begin{theorem}\label{space}
Space $\Omega$ is compact under the metric $\rho$.
\end{theorem}

\begin{remark} A proof of the above theorem is given in Appendix~\ref{appendix A}. A similar proof can be found in appendix~A of reference \cite{AM} to the case with a state space, consisting of a single fraction vector.
\end{remark}

\subsection{The mean field limit, sample paths}

In this subsection, we first discuss the stationary condition and present the generator for the Markov process $\displaystyle \{(\textbf{U}^{(N)}(t), \textbf{V}^{(N)}(t))\}$. We then prove that $\displaystyle \{(\textbf{U}^{(N)}(t), \textbf{V}^{(N)}(t))\}$ converges weakly to a limiting process, which is the solution to the mean field limit equations.

Since the sample paths of this system is the same as the classic Join-shortest-queue system when the queue size is greater than $M-1$, we give the stable regime of this queueing system in the following lemma.

\begin{lemma}
The system under the threshold-based workload control is stable if and only if $\lambda<\mu$.
\end{lemma}

Let $\mathbf{A}_{N}$ be the (infinitesimal) generator of the semigroup $T_N(t)$, $t\geq0$, associated with the process $\displaystyle \{(\textbf{U}^{(N)}(t), \textbf{V}^{(N)}(t))\}$.   Then, for a bounded continuous function $f: \Omega_{N} \rightarrow R $,  $\mathbf{A}_{N}f$ is described by:
\begin{equation}\label{Ant}
 \displaystyle\mathbf{A}_{N}f(\textbf{u}^N,\textbf{v}^N)=\sum\limits_{\displaystyle \scriptsize \begin{array}{c}
                                                            (\textbf{u}'^N,\textbf{v}'^N)\in\Omega_{N}, \\
                                                            (\textbf{u}'^N,\textbf{v}'^N)\neq (\textbf{u}^N,\textbf{v}^N)
                                                          \end{array}
}\displaystyle q_{(\textbf{u}^N,\textbf{v}^N)\longrightarrow(\textbf{u}'^N,\textbf{v}'^N)}(f(\textbf{u}'^N,\textbf{v}'^N)-f(\textbf{u}^N,\textbf{v}^N)),
\end{equation}
where $q_{(\textbf{u}^N,\textbf{v}^N)\longrightarrow(\textbf{u}'^N,\textbf{v}'^N)}$ is the transition rate from state $(\textbf{u}^N,\textbf{v}^N)$ to state $(\textbf{u}'^N,\textbf{v}'^N)$. Detailed information about $\mathbf{A}_{N}$ is needed in our analysis, which is provided in the following theorem. A proof to the theorem can be found in Aappendix~\ref{appendix B}.

\begin{theorem} \label{theorem 2.3}
Let $(\textbf{u}^N,\textbf{v}^N)\in\Omega_{N}$ be arbitrary state of process $\displaystyle \{(\textbf{U}^{(N)}(t), \textbf{V}^{(N)}(t)), t\geq0\}$. The generator $\textbf{A}_{N}$ of the Markov process $\displaystyle \{(\textbf{U}^{(N)}(t), \textbf{V}^{(N)}(t)), t\geq0\}$, acting on a bounded continuous function $f:\Omega_{N}\longrightarrow R$, is given by:
\begin{align}
 \nonumber \textbf{A}_{N}f(\textbf{u}^N,\textbf{v}^N)= &N\lambda\sum\limits_{k=2}\limits^{\infty}(u_{k-1}^{N}(t)-u_{k}^{N}(t)) W_{k}\displaystyle\left[f\displaystyle(\textbf{u}^{N}+\frac{\textbf{e}_{k}}{N},\textbf{v}^{N})-f(\textbf{u}^{N},\textbf{v}^{N})\right] \\
  \nonumber & + N\lambda v_{M-1}^{N}(t)W_{M}\left[f\displaystyle(\textbf{u}^{N}+\sum\limits_{k=1}\limits^{M}\frac{\textbf{e}_{k}}{N},\textbf{v}^{N}-\sum\limits_{k=0}
  \limits^{M-1}\frac{\textbf{e}_{0}}{N})-f(\textbf{u}^{N},\textbf{v}^{N})\right]\\
\nonumber &+  N\lambda\sum\limits_{k=1}\limits^{M-1}(v_{k-1}^{N}(t)-v_{k}^{N}(t)) W_{k} \displaystyle\left[f\displaystyle(\textbf{u}^{N},\textbf{v}^{N}+\frac{\textbf{e}_{k}}{N})-f(\textbf{u}^{N},\textbf{v}^{N})\right]\\
  \nonumber & +  N\mu(u_{1}^{N}(t)-u_{2}^{N}(t))\left[f\displaystyle(\textbf{u}^{N}-\frac{\textbf{e}_{1}}{N},\textbf{v}^{N}+\frac{\textbf{e}_{0}}{N})-
   f(\textbf{u}^{N},\textbf{v}^{N})\right]\\
 & + N\mu\sum\limits_{k=2}\limits^{\infty}(u_{k}^{N}(t)-u_{k+1}^{N}(t))\left[f\displaystyle(\textbf{u}^{N}-\frac{\textbf{e}_{k}}{N},\textbf{v}^{N})
-f(\textbf{u}^{N},\textbf{v}^{N})\right],
\end{align}
where $\textbf{e}_{k}$ is a vector whose $k$th entry is 1 and 0 elsewhere, $W_k$ is given in as follows:
\begin{equation*}
                W_{k}=\left\{\begin{array}{ll}
                               \sum\limits_{n=1}\limits^{d}C_{d}^{n}(v_{0}^{N}(t)-v_{1}^{N}(t))^{n-1}(v_{1}^{N}(t)+
                                            u_{1}^{N}(t))^{d-n}, &\hspace{-0.8cm}k=1, \\
                               \sum\limits_{n=1}\limits^{d}C^{n}_{d}(u_{k}^{N}(t)+v_{k}^{N}(t))^{d-n}
                                (u_{k-1}^{N}(t)-u_{k}^{N}(t)+v_{k-1}^{N}(t)-v_{k}^{N}(t))^{n-1},& \\
                                &\hspace{-0.8cm} k=2,3,\ldots,M-1,\\
                                \sum\limits_{n=1}\limits^{d}C^{n}_{d}(u_{M}^{N}(t))^{d-n}
                                (u_{M-1}^{N}(t)-u_{M}^{N}(t)+v_{M-1}^{N}(t))^{n-1}, &\hspace{-0.8cm} k=M, \\
                                \sum\limits_{n=1}\limits^{d}C^{n}_{d}(u_{k-1}^{N}(t)-u_{k}^{N}(t))^{n-1}
                                (u_{k}^{N})^{d-n},&\hspace{-0.8cm} k=M+1,M+2,\cdots.
                             \end{array}
                \right.
       \end{equation*}
\end{theorem}

In the following theorem, a system of mean field limit equations is constructed, whose solution is the approximation to the stochastic process $\{(\textbf{U}^{(N)}(t), \textbf{V}^{(N)}(t)), t\geq0\}$.
\begin{theorem}\label{limiting process}
If $\displaystyle(\textbf{U}^{N}(0),\textbf{V}^{N}(0))$ converges in distribution to a certain constant $\displaystyle(\textbf{c}_{u},\textbf{c}_{v})\in \Omega$ as $N\longrightarrow\infty,$
then the process $\displaystyle \{(\textbf{U}^{(N)}(t), \textbf{V}^{(N)}(t)),t\geq 0\}$ converges in distribution to a process $(\textbf{u}(t),\textbf{v}(t))\in\Omega$ as $N\longrightarrow\infty$. The process $(\textbf{u}(t),\textbf{v}(t))$ is given by the unique solution to the following system of differential equations:
\begin{eqnarray}
 \label{d1} \displaystyle(\textbf{u}(0),\textbf{v}(0)) &=& (\textbf{c}_{u},\textbf{c}_{v}), \\
  \label{d2}\displaystyle (\dot{\textbf{u}}(t),\dot{\textbf{v}}(t))&=&\textbf{y}(\textbf{u}(t),\textbf{v}(t)),
\end{eqnarray}
where 

\begin{eqnarray}
 \textbf{y}(\textbf{u}(t),\textbf{v}(t))&=& (y_{u_{1}} (\textbf{u},\textbf{v}),y_{u_{2}} (\textbf{u},\textbf{v}),\ldots;y_{v_{0}} (\textbf{u},\textbf{v}),y_{v_{1}} (\textbf{u},\textbf{v}),\ldots,y_{v_{M-1}} (\textbf{u},\textbf{v})), \\[4mm]
 \nonumber y_{u_{k}} (\textbf{u},\textbf{v})\hspace{-3mm} &=&\hspace{-3mm}\left\{\begin{array}{ll}
                                                \mu(u_{2}(t)-u_{1}(t))+\lambda v_{M-1}(t)W_{M}, &\hspace{-2mm}k=1,\\[4mm]
                                                \lambda (u_{k-1}(t)-u_{k}(t))W_{k}+\mu(u_{k+1}(t)-u_{k}(t))+\lambda v_{M-1}(t)W_{M},&\hspace{-2mm}k=2,3,\ldots,M, \\[4mm]
                                                \lambda (u_{k-1}(t)-u_{k}(t))W_{k}+\mu(u_{k+1}(t)-u_{k}(t)),&\hspace{-2mm} k=M+1,\ldots,
                                              \end{array}
                                              \right.\\[4mm]
                                              \\[4mm]
  y_{v_{k}} (\textbf{u},\textbf{v}) \hspace{-3mm}&=& \hspace{-3mm}\left\{\begin{array}{ll}
                                                 \mu(u_{1}(t)-u_{2}(t))-\lambda v_{M-1}(t)W_{M}, & ~~~~~~~~~~~~~~k=0, \\[4mm]
                                                 \lambda (v_{k-1}(t)-v_{k}(t))W_{k}-\lambda v_{M-1}(t)W_{M}, &~~~~~~~~~~~~~~ k=1,2,3,\ldots,M-1,
                                               \end{array}
                                        \right.
\end{eqnarray}
\end{theorem}

Before proving Theorem~\ref{limiting process}, we provide the following two lemmas, proofs to which are given in Appendix~\ref{appendix C}.

\begin{proposition}\label{p2.1}
If $(\textbf{c}_{u},\textbf{c}_{v})\in \Omega$, then the system of mean field limiting equations \eqref{d1}--\eqref{d2} has a unique solution.
\end{proposition}

\begin{proposition}\label{pd}
Let $S_1=\{(k,i,j)\mid k\geq 1;i\geq 1; 0\leq j\leq M-1\}$ and $S_2=\{(k,i,i',j,j')\mid k\geq 1; i,i'\geq 1; 0\leq j,j'\leq M-1\}$. Denote by $(\textbf{u}(t,\textbf{c}_{u}),\textbf{v}(t,\textbf{c}_{v}))$  the unique solution to \eqref{d1}--\eqref{d2}. Then, all the first-order and second-order partial derivatives (appeared below) exist and satisfy the following inequalities:
\begin{eqnarray*}
  \label{pd1} &&\sup\limits_{(k,i,j)\in S_1}\left\{\mid \frac{\partial u_{k}}{\partial c_{ui}}\mid,\mid \frac{\partial u_{k}}{\partial c_{vj}}\mid\right\}\leq\exp \left\{At\right\}, \\[4mm]
  \label{pd2} &&\sup\limits_{(k,i,j)\in S_1}\left\{\mid \frac{\partial v_{k}}{\partial c_{ui}}\mid,\mid \frac{\partial v_{k}}{\partial c_{vj}}\mid\right\}\leq\exp\left\{A't\right\},\\[4mm]
  \label{pd3}&&\sup\limits_{(k,i,i',j,j')\in S_2
  }\left\{\mid\frac{\partial^2u_k}{\partial c_{ui}\partial c_{ui'}}\mid,\mid\frac{\partial^2u_k}{\partial c_{ui}\partial c_{vj}}\mid,\mid\frac{\partial^2u_k}{\partial c_{vj}\partial c_{vj'}}\mid\right\}\leq\frac{B}{A+2\mu}\left[\exp\{2At\}-\exp\{At\}\right],\\[4mm]
  &&\sup\limits_{(k,i,i',j,j')\in S_2
  }\left\{\mid\frac{\partial^2v_k}{\partial c_{ui}\partial c_{ui'}}\mid,\mid\frac{\partial^2u_k}{\partial c_{ui}\partial c_{vj}}\mid,\mid\frac{\partial^2u_k}{\partial c_{vj}\partial c_{vj'}}\mid\right\}\leq\frac{B}{A}\left[\exp\{2A't\}-\exp\{A't\}\right],\\[4mm]
\end{eqnarray*}
where $c_{ui}$ and $c_{vi}$ denote the $i$-th element of $\textbf{c}_{u}$ and $\textbf{c}_{v}$, respectively. Besides, all constants $A$, $A'$, $B$ and $K_i$ for $i=1, 2, \ldots, 6$ are given by 
\begin{align}
  K_1 =& \sum\limits_{n=1}\limits^{d}C_{d}^{n}2^{n-1}2^{d-n}=2^{d-1}(2^d-1),  \label{eqn:K1} \\
  K_2 =& \sum\limits_{n=1}\limits^{d}C_{d}^{n}(d-n)2^{n}2^{d-n}+\sum\limits_{n=1}\limits^{d}C_{d}^{n}(n-1)2^{n}2^{d-n}=(d-1)2^{d}(2^d-1),\\
  K_3 =& \sum\limits_{n=1}\limits^{d}C_{d}^{n}2^{n-1}=\frac{1}{2}(3^d-1)\\
  K_4 =& \sum\limits_{n=1}\limits^{d}C_{d}^{n}(d-n)2^{n}+\sum\limits_{n=1}\limits^{d}C_{d}^{n}(n-1)2^{n}=(d-1)(3^d-1),\\
  K_5 =& 4\sum\limits_{n=1}\limits^{d}C_{d}^{n}(d-n)(d-n-1)2^{d-n-2}2^{n-1}+16\sum\limits_{n=1}\limits^{d}C_{d}^{n}(d-n)(n-1)2^{d-n-1}
  2^{n-2} \nonumber \\
   & + 16\sum\limits_{n=1}\limits^{d}C_{d}^{n}(n-1)(n-2)2^{d-n}2^{n-3} \nonumber \\
  =& 2^{d-1}\sum\limits_{n=1}\limits^{d}C_{d}^{n}(d-n)(d-n-1)+2^{d+1}(d-2)\sum\limits_{n=1}\limits^{d}C_{d}^{n}(n-1),\\
  K_6=& \sum\limits_{n=1}\limits^{d}C_{d}^{n}(d-n)(d-n-1)2^{n-1}+6\sum\limits_{n=1}\limits^{d}C_{d}^{n}(d-n)(n-1)2^{n-1} \nonumber \\
      &+9\sum\limits_{n=1}\limits^{d}C_{d}^{n}(n-1)(n-2)2^{n-3}, \label{eqn:K6}
\end{align}
and $A=(2K_1+K_2+K_3+K_4)\lambda+2\mu$, $A'=(2K_1+K_2+K_3+K_4)\lambda$, $B=(2K_2+2K_5+K_4+K_6)\lambda$.
\end{proposition}

We are now ready to provide a proof to Theorem~\ref{limiting process}.

\begin{proof}[\textit{Proof to Theorem~\ref{limiting process}.}]
 Let $L$ denote the set of the real continuous functions on space $\Omega$. Let $D$ be a subset of $L$, containing the uniformly bounded functions whose partial derivatives satisfy the inequalities given in Proposition~\ref{pd}. Then, according to Proposition 2 and Lemma 15 in \cite{VD}, we know that $D$ is a core of $A$ by constructing $D_0\subset D$ , which depends on finitely many variables.
 Moreover, for $f\in D$, we have
\begin{eqnarray*}
  N \displaystyle\left[f\displaystyle(\textbf{u}+\frac{\textbf{e}_{k}}{N},\textbf{v})-f(\textbf{u},\textbf{v})\right]&\longrightarrow& \frac{\partial f(\textbf{u},\textbf{v})}{\partial u_k},\\
  N\displaystyle\left[f\displaystyle(\textbf{u},\textbf{v}+\frac{\textbf{e}_{j}}{N})-f(\textbf{u},\textbf{v})\right] &\longrightarrow&  \frac{\partial f(\textbf{u},\textbf{v})}{\partial v_j},\\
  N\left[f\displaystyle(\textbf{u}-\frac{\textbf{e}_{k}}{N},\textbf{v}+\frac{\textbf{e}_{k}}{N})-
   f(\textbf{u},\textbf{v})\right]&\longrightarrow&\frac{\partial f(\textbf{u},\textbf{v})}{\partial v_k}-\frac{\partial f(\textbf{u},\textbf{v})}{\partial u_k}.
\end{eqnarray*}

Thus,
\begin{eqnarray}\label{An}
 \nonumber \lim\limits_{N\longrightarrow\infty}\textbf{A}_{N}f(\textbf{u}^N,\textbf{v}^N)&=&\lambda\sum\limits_{k=2}\limits^{\infty}(u_{k-1}^{N}(t)-u_{k}^{N}(t)) W_{k}\displaystyle\frac{\partial f(\textbf{u},\textbf{v})}{\partial u_k}\\
  \nonumber &&\hspace*{-3.6cm} + \lambda v_{M-1}^{N}(t) V_{M}\displaystyle\left(\sum\limits_{k=1}\limits^{M}\frac{\partial f(\textbf{u},\textbf{v})}{\partial u_k}-\sum\limits_{k=0}\limits^{M-1}\frac{\partial f(\textbf{u},\textbf{v})}{\partial v_k}\right)
  + N\lambda\sum\limits_{k=1}\limits^{M-1}(v_{k-1}^{N}(t)-v_{k}^{N}(t)) V_{k} \displaystyle\frac{\partial f(\textbf{u},\textbf{v})}{\partial v_k}\\
  \nonumber &&\hspace*{-3.6cm} + N\mu(u_{1}^{N}(t)-u_{2}^{N}(t))\left[\frac{\partial f(\textbf{u},\textbf{v})}{\partial v_0}-\frac{\partial f(\textbf{u},\textbf{v})}{\partial u_1}\right]
   - N\mu\sum\limits_{k=2}\limits^{\infty}(u_{k}^{N}(t)-u_{k+1}^{N}(t))\frac{\partial f(\textbf{u},\textbf{v})}{\partial u_k}\\[4mm]
&=&\displaystyle\frac{\partial}{\partial t}f\left(\textbf{u}(t,\textbf{c}_u),\textbf{v}(t,\textbf{c}_v)\right)\mid_{t=0}.
\end{eqnarray}

It is clear that if $T(t)$ is the semigroup of the operators associated with $(\textbf{u}(t),\textbf{v}(t))$. Then,
$T(t)f(\textbf{u},\textbf{v})=f(\textbf{u}(t,\textbf{c}_u),\textbf{v}(t,\textbf{c}_v))$ and its corresponding generator $A$ satisfies that \[Af(\textbf{u},\textbf{v})=\lim\limits_{t\longrightarrow0}\frac{T(t)f(\textbf{u},\textbf{v})-f(\textbf{u},\textbf{v})}{t}=\displaystyle\frac{\partial}{\partial t}f\left(\textbf{g}(t,\textbf{c}_u),\textbf{h}(t,\textbf{c}_v)\right)\mid_{t=0},\]
which indicates $A_N f \longrightarrow A f$ for all $f\in D$.

By Theorem 6.1 (p.28) of \cite{EK}, it follows that $T_N (t)\to T(t)$ as $N$ goes to infinity. Finally, with the help of  Theorem 2.11 of \cite{EK}, we have the desired result.
\end{proof}


\subsection{Mean field limit, stationary behaviours}

In this subsection, we prove that the stationary distribution of $\{(\textbf{U}^{(N)}(t), \textbf{V}^{(N)}(t))\}$ converges weakly to $\delta_{\pi}$, where $\pi$ is the stationary state of the limiting process and $\delta$ is the Dirac function. Some of the important stationary measures for this queueing system are also presented.

Due to the construction of the process, $(\mathbf{\pi}_{u},\mathbf\pi_{v})\triangleq\lim\limits_{t\to \infty}(\mathbf{u}(t,\mathbf{c}_u),\mathbf{v}(t,\mathbf{c}_v))$ can be seen as either the stationary state of $(\mathbf{u}(t),\mathbf{v}(t))$ or the limiting distribution of process $\{X(t),I(t)\}$, where $X(t)$ and $I(t)$ denote the queue length and the server's status, respectively. Under the ergodic condition, we know that the limiting distribution of the Markov process is exactly the stationary distribution of the process.

 It is worth noting that $(\pi^{W},\pi^{V})Q=0$ is exactly the same as setting the mean field equations equal to zero. Specifically,
 \begin{eqnarray}
   \label{1} \mu(\pi_{2}^{W}-\pi_{1}^{W})+\lambda \pi_{M-1}^{V}W_{M}&=&0,\\[4mm]
   \label{2} \lambda (\pi_{k-1}^{W}-\pi_{k}^{W})W_{k}+\mu(\pi_{k+1}^{W}-\pi_{k}^{W})+\lambda \pi_{M-1}^{V}W_{M}&=&0,~~~ k=2,3,\cdots,M, \\[4mm]
   \label{3}\lambda (\pi_{k-1}^{W}-\pi_{k}^{W})W_{k}+\mu(\pi_{k+1}^{W}-\pi_{k}^{W})&=&0, ~~~k=M+1,M+2,\ldots,\\[4mm]
   \label{4}\mu(\pi_{1}^{W}-\pi_{2}^{W})-\lambda \pi_{M-1}^{V}W_{M}&=&0, \\[4mm]
   \label{5} \lambda (\pi_{j-1}^{V}-\pi_{j}^{V})W_{j}-\lambda \pi_{M-1}^{V}W_{M}&=&0,~~~ j=1,2,3,\ldots,M-1.
  \end{eqnarray}

In the following theorem, we will state that the set of equations \eqref{1}--\eqref{5} has a unique solution $(\pi^{W},\pi^{V})$, or equivalently the process $\{X(t),I(t)\}$ has a unique stationary distribution.

%
\begin{theorem}\label{fixed point}
There exists a unique stationary distribution for the process $\{X(t),I(t)\}$, the unique solution to \eqref{d1}--\eqref{d2} in the space $\Omega$, which can be recursively computed.
\end{theorem}

\begin{proof}
Simplifying $W_k$ $(k=1,2,3,\ldots)$, we have
\[W_k=\left\{\begin{array}{ll}
               \displaystyle \frac{(u_1+v_0)^d-(u_1+v_1)^d}{v_0-v_1}, & k=1, \\[4mm]
               \displaystyle \frac{(u_{k-1}+v_{k-1})^d-(u_k+v_k)^d}{u_{k-1}+v_{k-1}-u_k-v_k}, & k=2,3,\ldots,M-1, \\[4mm]
               \displaystyle \frac{(u_{M-1}+v_{M-1})^d-u_M^d}{u_{M-1}+v_{M-1}-u_M},& k=M, \\[4mm]
               \displaystyle \frac{u_{k-1}^d-u_k^d}{u_{k-1}-u_k}, & k=M+1,M+2,\ldots.
             \end{array}
\right.
\]

Summing all the equations \eqref{1}-\eqref{5} with all $k$ gives
       \[\lambda(\pi_{0}^{V}+\pi_{1}^{W})^d-\mu \pi_{1}^{W}=0.\]
Using the normalization condition $\pi_{0}^{V}+\pi_{1}^{W}=1$, we have
\[\pi_{1}^{W} =\frac{\lambda}{\mu},~~~~\pi_{0}^{V} = 1- \frac{\lambda}{\mu}.\]

To see $\pi$ can be recursively computed, we first assume that $\xi_0=1- \frac{\lambda}{\mu}$ and $\delta_{1}=\frac{\lambda}{\mu}$. Then, we take that $\pi_{0}^{V}=\xi_0$, $\pi_{1}^{W}=\delta_{1}$ and $\pi_{2}^{W}=\delta_{2}$ (to be determined later). By $\xi_{1}$, we denote the solution in $(0,\xi_0)$ to the nonlinear equation
\begin{equation*}
  G_{1}(x)=\lambda(\xi_{0}-x)\overline{V}_{1}(\xi_{0},x;\delta_1)+\mu(\delta_2-\delta_1)=0,
\end{equation*}
and $\xi_{2}$ the solution in $(0,\xi_{1})$ to
\begin{equation*}
  G_{2}(x)=\lambda(\xi_{1}-x)\overline{V}_{2}(\xi_{1},x;\delta_1,\delta_2)+\mu(\delta_2-\delta_1)=0.
\end{equation*}
We then have
\begin{equation*}
  \delta_{3}=\frac{\lambda}{\mu}(\delta_2-\delta_1)\overline{W}_{2}(\xi_1,\xi_2;\delta_1,\delta_2)+(2\delta_2-\delta_1).
\end{equation*}
Following this iterative algorithm, we can have $(\xi_0,\xi_1,\cdots,\xi_{M-1})$ and $(\delta_{1},\delta_{2},\cdots,\delta_{M})$, which contains only one unknown parameter $\delta_2$ that can be determined by applying equation \eqref{1}.
Moreover,
\begin{equation*}
  \delta_{k+1}=\frac{\lambda}{\mu}(\delta_{k}-\delta_{k-1})W_{k}(\delta_{k-1},\delta_{k}),~~~k=M,M+1,\ldots.
\end{equation*}

Now, set
\begin{eqnarray*}
\pi_{k}^{V} &=& \xi_{k},~~~k=0,1,\ldots,M-1, \\
\pi_{j}^{W} &=& \delta_{j},~~k=1,2,\ldots,
\end{eqnarray*}
which is a globally asymptotically stable fixed point for the limiting process.

We finally prove the uniqueness of $\pi$.
Substituting $x=0$ in $G_1(x)$ and equation \eqref{1}, we have
\begin{eqnarray*}
  G_{1}(0)&=&\lambda\left[((\xi_0+\delta_1)^d-(\delta_1)^d)-\frac{\xi_{M-1}}{\delta_{M-1}-\delta_{M}+\xi_{M-1}}((\xi_{M-1}+\delta_{M-1})^d-(\delta_M)^d)\right]\\
         &=&\lambda[(\xi_0\sum\limits_{k=0}\limits^{d-1}(\xi_0+\delta_1)^k(\delta_1)^{d-1-k})-(\xi_{M-1}\sum\limits_{k=0}\limits^{d-1}(\xi_{M-1}+\delta_{M-1})^k(\delta_M)^{d-1-k}))]\\
         &\geq&\lambda\xi_0\sum\limits_{k=0}\limits^{d-1}[(\xi_0+\delta_1)^k(\delta_1)^{d-1-k}-(\xi_{M-1}+\delta_{M-1})^k(\delta_M)^{d-1-k}].
\end{eqnarray*}
Since $\xi_{0}\geq\xi{M-1}$ and $\delta_1\geq\delta_{M-1}\delta_{M}$, $G_1(0)\geq 0$.
Also,
\begin{equation*}
  G_1(\xi_{0})=\mu(\delta_{2}-\delta_{1})<0.
\end{equation*}
Differentiating $G_1(x)$, we can derive
\begin{equation*}
  G'_1(x)=-\lambda(x+\delta_1)^{d-1}<0,
\end{equation*}
which means that $G_1(x)$ is monotonically decreasing.
Since $G_{1}(\xi_1)=0$, $\xi_{1}$ is the unique solution to $G_1(x)$ in the interval $(0, \xi_0)$.
Similarly, we can prove that $\xi_{k}$ is the unique solution to $G_k(x)$ in interval $(0,\xi_{k-1})$.
\end{proof}

\begin{remark}
We obtained the (unique) stationary distribution $(\mathbf{\pi}^{W},\mathbf\pi^{V})$ for the process $\{X(t),I(t)\}$ and emphasized that under the ergodic condition, the limiting distribution of $\{X(t),I(t)\}$ is $\lim\limits_{t\to \infty}(\mathbf{u}(t,\mathbf{c}_u),\mathbf{v}(t,\mathbf{c}_v))=(\mathbf{\pi}^{W},\mathbf\pi^{V})$, based on which it is clear that $(\mathbf{\pi}^{W},\mathbf\pi^{V})$ is also the globally asymptotically stable fixed point for the mean field equations.
\end{remark}


The following theorem confirms that the stationary distribution for the original system can be approximated by the stationary distribution of the limiting process.
\begin{theorem}
\begin{equation*}
 \lim\limits_{N\longrightarrow\infty}\widetilde{\pi}_{N}=\delta_{\pi},
\end{equation*}
or equivalently $(\mathbf{u}^N(\infty),\mathbf{v}^N(\infty))\to (\mathbf{\pi}^{W},\mathbf\pi^{V})$, where $\widetilde{\pi}_{N}$ is the stationary distribution of $\displaystyle \{(\textbf{U}^{(N)}(t), \textbf{V}^{(N)}(t)),t\geq 0\}$.
\end{theorem}

\begin{proof}
We have proved that the space $\Omega$ is compact under the metric $\rho$. Hence, the sequence of probability measures $(\widetilde{\pi}_N)_N$ is tight. Thus, $(\widetilde{\pi}_N)_N$ is relatively compact by Theorem 2.2 (p.104) of \cite {EK}, and then has limit points. In order to prove this theorem, we now only need to show that all convergent subsequences share the same limiting points.

Let $(\widetilde{\pi}_{N_k})_k$ denote an arbitrary subsequence of $(\widetilde{\pi}_N)_N$, which converges to the distribution $\widetilde{\pi}$. Then, let the process $(\textbf{U}^{N_k}(t),\textbf{V}^{N_k}(t))$, $(\textbf{u}(t),\textbf{v}(t))$ start from stationary state $\widetilde{\pi}_{N_k}$ and $\widetilde{\pi}$, respectively. According to Theorem \ref{limiting process}, it follows that $(\textbf{U}^{N_k}(t),\textbf{V}^{N_k}(t))\longrightarrow(\textbf{u}(t),\textbf{v}(t)),$ so $\widetilde{\pi}_{N_k}\longrightarrow\delta_{\pi}$. Together with $\widetilde{\pi}_{N_k}\longrightarrow\widetilde{\pi}$, we have $\widetilde{\pi}_{N_k}=\delta_{\pi}$. Thus, the theorem is proved.
\end{proof}

Before ending this section, we provide useful expressions, based on the stationary distribution for the limiting process, which are approximations to the respective stationary measures for the original queueing system under the stability condition $\lambda<\mu$.

The mean queue length when a machine is working:
\begin{equation}
  E(Q_W)=\sum\limits_{k=1}\limits^{\infty}\pi_{k}^{W}=\sum\limits_{k=1}\limits^{\infty}\delta_k.
\end{equation}

The mean queue length when a machine is dormant:
\begin{equation}
  E(Q_D)=\sum\limits_{k=1}\limits^{\infty}\pi_{k}^{V}=\sum\limits_{k=1}\limits^{M-1}\xi_k.
\end{equation}

The mean length:
\begin{eqnarray}
  \nonumber E(Q)&=&\sum\limits_{k=1}\limits^{\infty}P(Q\geq k)\\
  &=&\sum\limits_{k=1}\limits^{M-1}(\pi_k^W+\pi_k^V)+\sum\limits_{k=M}^{\infty}\pi_k^W=\sum\limits_{k=1}\limits^{\infty}\delta_k+
  \sum\limits_{k=1}\limits^{M-1}\xi_k.
\end{eqnarray}

The mean sojourn time of a task at a working virtual machine:
\begin{equation}
E(S_W)=\sum\limits_{k=2}\limits^{\infty}k\mu\frac{(\delta_{k-1}-\delta_{k})W_{k}\mu}{\lambda}.
\end{equation}

The mean sojourn time of a task at a dormant virtual machine:
\begin{equation}
  E(S_V)=\sum\limits_{k=1}\limits^{M-1}(\frac{M-k}{\lambda}+k\mu)\frac{(\xi_{k-1}-\xi_{k})W_{k}\mu}{\mu-\lambda}+\frac{M\mu^2\xi_{M-1}W_{M}}{\mu-\lambda}.
\end{equation}

The mean sojourn time of an arbitrary task in this system:
\begin{equation}
 E(S)=\sum\limits_{k=2}\limits^{\infty}k\mu(\delta_{k-1}-\delta_{k})W_{k}+\sum\limits_{k=1}\limits^{M-1}(\frac{M-k}{\lambda}+k\mu)(\xi_{k-1}-\xi_{k})W_{k}+M\mu\xi_{M-1}W_{M}.
\end{equation}

\section{Numerical analysis}

In this section, we provide some numerical examples to show differences of mean queue length between simulation results (true values) and mean field limit value (approximations), as well as the impacts of system parameters on the system performance. Unless otherwise specified, parameters are set by the following default: $\mu=1$ (for convenience), $M=2$ for easy calculations, $\lambda=0.39$ (a moderate traffic case), and $d=2$ (which is not crucial and a similar pattern can be observed for larger $d$).

\begin{figure}
  \centering
  \includegraphics[width=13cm,height=8cm]{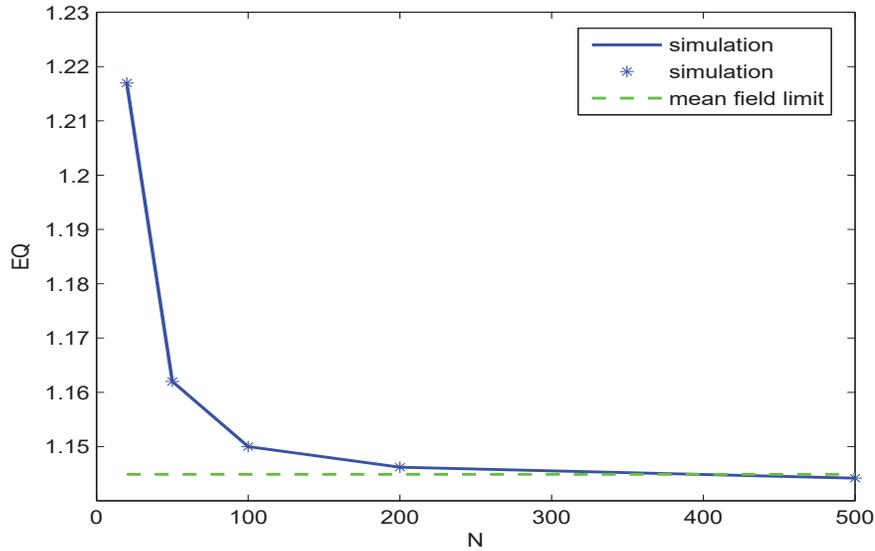}
  \caption{The difference for the average queue length between simulation results, for various $N$, and the mean field limit}\label{simulation}
\end{figure}
We simulate the model with $N=20,50,100,200,500$, respectively, and the simulation results are shown in Fig.~\ref{simulation}, together with the approximations based on the mean field limit. As expected, the difference between them becomes smaller as $N$ becomes larger. In fact, we can see that when $N=200$ the approximation is very close to the real value ( relative error $< 10\%$, when $N=20$; relative error $< 1\%$, when $N=100$).

\begin{figure}
\centering
\includegraphics[width=13cm,height=8cm]{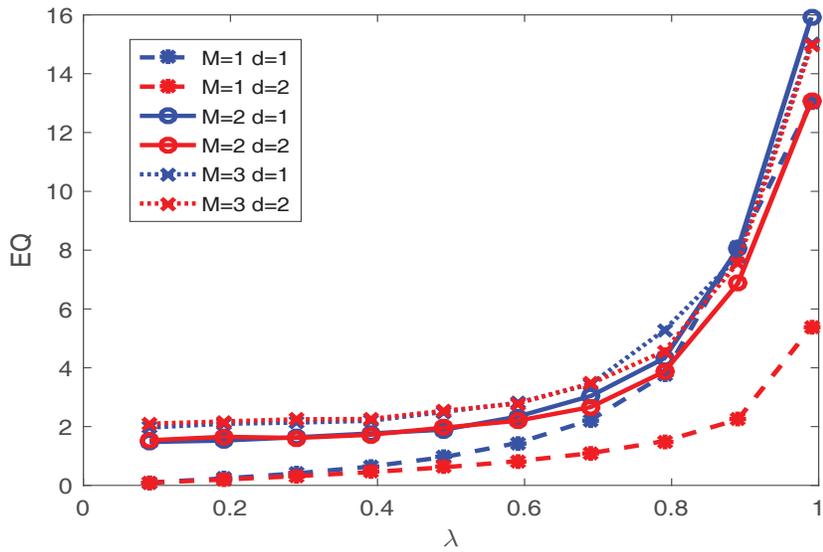}
\caption{The mean queue length versus $\lambda$}\label{11}
\end{figure}
\begin{figure}
\centering
\includegraphics[width=13cm,height=8cm]{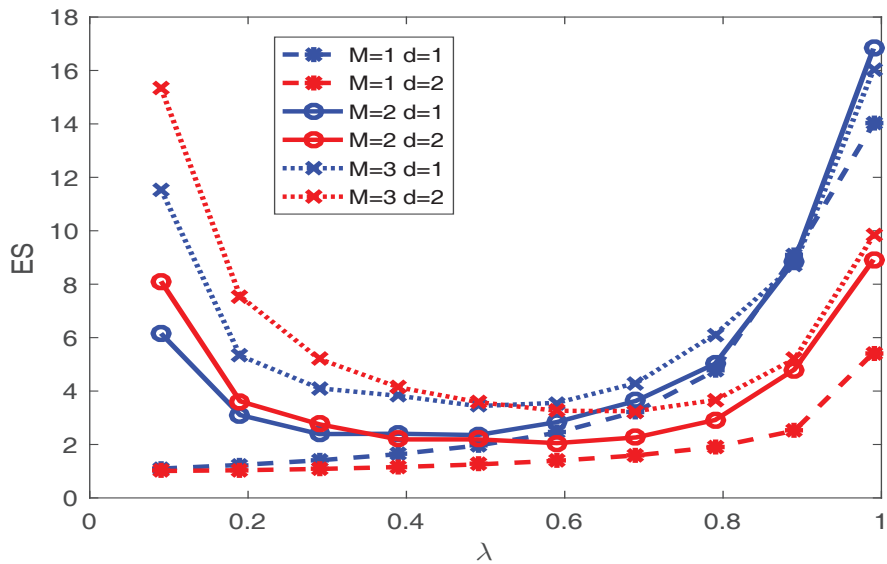}
\caption{The mean sojourn time versus $\lambda$}\label{22}
\end{figure}

Fig.~\ref{11} and Fig.~\ref{22} depict how the mean queue length and mean sojourn time depend on the the threshold value $M$ and sample size $d$, for various values of $\lambda\in[0.09,0.99]$. Fig.~\ref{11} shows that $EQ$ increases with $\lambda$ for any fixed $d$. Moreover, $EQ$ declines when $d$ becomes larger for any fixed $M$. Fig.~\ref{22} illustrates that $ES$ becomes larger with larger $M$ for a fixed value of $d$, as expected. For $M\geq2$, it is interesting enough to note that the mean sojourn time decreases to the minimum as $\lambda$ increases to a certain point, after that the sojourn time increases as $\lambda$, due to the fact that more and more virtue machines become dormant as $\lambda$ becomes smaller and smaller, and a longer time is needed to reach the threshold value, a phenomenon we do not have for the system with $M=1$. However, as $\lambda$ become large enough, the time needed to reach the threshold decreases to zero, and then the mean sojourn time has a positive correlation with $\lambda$, a phenomenon consistent with that for the system with $M=1$. Moreover, when $\lambda$ is small, the sojourn time can be longer for bigger $d$, since when the system is almost empty, the bigger $d$ means an arrival task joins a dormant virtual machine with a greater possibility, then dormant machines need more time to be activated. This  phenomenon will disappear as $\lambda$ increase to a certain point, after which the sojourn time is always smaller when $d$ is bigger. So, curves for $d=1$ and $M=2$, and for $d=2$ and $M=2$ cross each other at a certain $\lambda$ value.

\begin{figure}
\centering
\includegraphics[width=13cm,height=8cm]{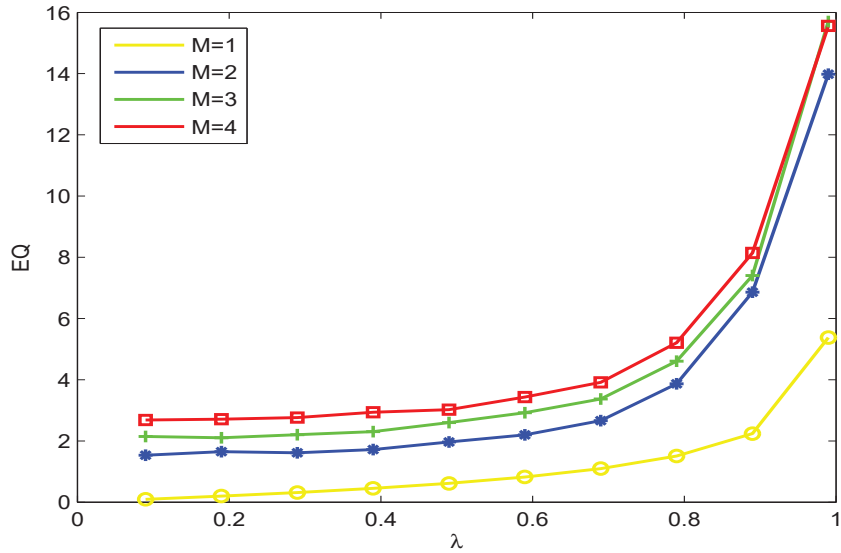}
\caption{The mean queue length versus $\lambda$}\label{33}
\end{figure}
\begin{figure}
\centering
\includegraphics[width=13cm,height=8cm]{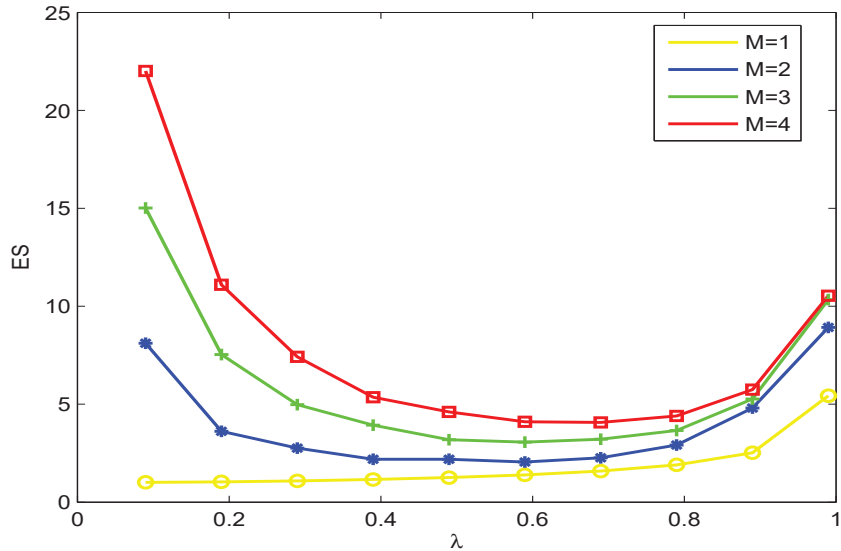}
\caption{The mean sojourn time versus $\lambda$}\label{44}
\end{figure}

Fig.~\ref{33} and Fig.~\ref{44} demonstrate how $\lambda$ impacts on the two performance measures, $EQ$ and $ES$, respectively, for different threshold values.
Fig.~\ref{33} shows that $EQ$ increases with $\lambda$ and $M$, respectively. Fig.~\ref{44} demonstrates that $ES$ becomes longer as $M$ becomes larger. Besides, when $M$ becomes larger, the $\lambda$ corresponding to the shortest sojourn time increases.

In Table.~1 and Table.~2, we provide the optimal $M$, which is chosen according to the system performance measures: the mean queue length and the mean sojourn time, respectively. For example, the second column in Table.~1 shows that if the mean queue length is required to be at most $4$, the corresponding optimal threshold value is $5$ (or the maximal value of $M$). Similarly, the figures given in Table.~2 are the optimal values for $M$ within the maximal acceptable mean sojourn time.
\[\begin{array}{cc}
\text{\small Table 1. the optimal M with maximum EQ}~~~&\text{\small Table 2. the optimal M with maximum ES}\\
\begin{tabular}{cccccc}
\hline\hline
$EQ$ & 4 & 6 & 8 &10 & 12 \\ \hline
$M$ & 5 & 8& 12 & 14 & 18 \\
\hline\hline
\end{tabular}
 ~~~ &
\begin{tabular}{cccccc}
\hline\hline
$ES$ & 6 & 9& 12 &15 & 18 \\ \hline
$M$ & 5 & 9 & 11 & 15 & 20 \\
\hline\hline
\end{tabular}
\end{array}\]

\section{Conclusion}

This paper studies a homogeneous queueing system with threshold-based workload control scheme. We first set up a Markov process for this system and construct an infinite-dimensional system of mean field limit equations. we then apply convergence theories to prove that the above Markov process converges to the solution to the mean field limit equations. In addition, we provide an iterative algorithm for computing the fixed point of the deterministic system which is indeed the stationary behaviour of the limiting process. Finally, numerical examples and simulation results for performance measures are also presented, which can be used for the purpose of system design and optimization. This queueing system saves $(1-\frac{\lambda}{\mu})$\% of the energy consumption compared to the ordinary supermarket models. On the other hand, during the idle time if virtue machines can be  used for other tasks, the machine utilization will be improved significantly.

There are several avenues for further research. We have proved that the Markov process for this homogeneous system converges to a limiting process, and in the future study, one could investigate the convergence rate of this Markov process. The algorithm for computing the fixed point could be further improved to be more effective. There are also many model extensions worthy a consideration. For example, one could consider a general heterogeneous queueing system with a non-exponential service times and a non-Poisson arrival process.

\section*{Acknowledgement}
This work was supported in part by the National Natural Science Foundation of China (No.61773014) and the Natural Sciences and Engineering Research Council of Canada.

\appendix

\section{Proof to Theorem~\ref{space}}
\label{appendix A}


\begin{proof}
 Proving the compactness of space $\Omega$ under metric $\rho$ is equivalent to prove that any sequence $\left\{(\textbf{u},\textbf{v})_{n}, n\in Z_+\right\}\subset \Omega$ has a convergent subsequence $\left\{(\textbf{u},\textbf{v})_{n_m},n_m\in Z_+\right\}$ and the convergence point belongs to $\Omega$ under the metric $\rho$. Observing that all the elements of vector $(\textbf{u},\textbf{v})$ are bounded and less than one,  \red{so,} $(\textbf{u},\textbf{v})_{n}$ must have a convergent subsequence $(\textbf{u},\textbf{v})_{n_m}$, $(\textbf{u},\textbf{v})_{n_m}\rightarrow(\textbf{u},\textbf{v})$ and $(\textbf{u},\textbf{v})\in \Omega$ in the sense of absolute value.

Then, we need to show that $\displaystyle \rho\left\{(\textbf{u},\textbf{v})_{n_m}, (\textbf{u},\textbf{v})\right\}\longrightarrow0$ as $m\longrightarrow\infty$.
We choose am $m$ large enough such that
\[\displaystyle\mid\frac{u_{n_m}(k)-u(k)}{k+1}\mid\leq\frac{1}{l+1},~~~~~ \displaystyle\mid\frac{v_{n_m}(j)-u(j)}{j+1}\mid\leq\frac{1}{l+1},\]
 for $0\leq k,j\leq l$. When $k,j>l$,
 \[\displaystyle\mid\frac{u_{n_m}(k)-u(k)}{k+1}\mid\leq\frac{1}{l+1},~~~~~ \displaystyle\mid\frac{v_{n_m}(j)-u(j)}{j+1}\mid\leq\frac{1}{l+1}\]
  also holds, since $\mid u_{n_m}(k)-u(k)\mid\leq 1$, $\mid v_{n_m}(j)-v(j)\mid\leq 1$ and $k,j>l$.

 Therefore,
\[\displaystyle \rho\left\{(\textbf{u},\textbf{v})_{n_m}, (\textbf{u},\textbf{v})\right\}\leq\frac{1}{l+1},\]
implying that $\displaystyle \rho\left\{(\textbf{u},\textbf{v})_{n_m}, (\textbf{u},\textbf{v})\right\} \longrightarrow 0$ as $l\longrightarrow\infty$. So, the space $\Omega$ is compact under the metric $\rho$.
\end{proof}

\section{Proof to Theorem~\ref{theorem 2.3}}
\label{appendix B}

\begin{proof}
 There are two processes in this system: the arrival process and the service process. We proceed with our proof accordingly. In the proof, we use the abbreviation VM for virtual machine.

\begin{description}
  \item[Arrival process: joining a working VM]
  In this process, there are three different parts divided according to the number of tasks in a virtual machine and its buffer. 
        \begin{description}
            \item [$\mathbf{k=2,3,\ldots,M-1}$]
                This part contains three possibilities.
                \begin{description}

                    \item[Possibility \uppercase\expandafter{\romannumeral 1}] All sampled VMs are in working state, and there is at least one VM with $k-1$ tasks, while the others are with at least $k$ tasks. An arriving task is assigned to any one of the working VM with $k-1$ tasks. The probability of this situation is given by
                        \begin{equation}
                       \hspace*{-1.8cm} \sum\limits_{n=1}\limits^{d}C^{n}_{d}(u_{k-1}^{N}(t)-u_{k}^{N}(t))^{n}(u_{k}^{N})^{d-n}
                                 = (u_{k-1}^{N}(t)-u_{k}^{N}(t))\times  \displaystyle\sum\limits_{n=1}\limits^{d}C^{n}_{d}(u_{k-1}^{N}(t)-u_{k}^{N}(t))^{n-1}(u_{k}^{N})^{d-n}.
                        \end{equation}

                    \item[Possibility \uppercase\expandafter{\romannumeral 2}] There are both working and dormant VMs in the sample. There are at least one working VM with $k-1$ tasks, while other working and dormant VMs are all with at least $k$ tasks. An arriving task is assigned to any one of the working VM with $k-1$ tasks. The probability of this situation is given by,
                            \begin{eqnarray}
                                \nonumber&&\displaystyle\sum\limits_{n=1}\limits^{d-1}C^{n}_{d}(u_{k-1}^{N}(t)-u_{k}^{N}(t))^{n}
                                \sum\limits_{m=1}\limits^{d-n}(v_k^N(t))^m(u_{k}^{N})^{d-n-m}  = \\
                                 && \hspace*{-2.5cm} (u_{k-1}^{N}(t)-u_{k}^{N}(t)) \displaystyle\sum\limits_{n=1}\limits^{d-1}C^{n}_{d}(u_{k-1}^{N}(t)-u_{k}^{N}(t))^{n-1}
                                        \left[(v_k^N(t)+u_k^N(t))^{(d-n)}-(u_{k}^{N})^{d-n}\right]
                            \end{eqnarray}

                    \item[Possibility \uppercase\expandafter{\romannumeral 3}] This situation also contains both working and dormant VMs. There is at least one working VM with $k-1$ tasks and at least one dormant VM with $k-1$ tasks, while other VMs (either working or dormant) with at least $k$ tasks. An arriving task is assigned to any one of the working VM with $k-1$ tasks randomly. The probability of this situation is given by
                            \begin{eqnarray}
                                \nonumber&&\displaystyle\sum\limits_{n=2}\limits^{d}C^{n}_{d}\sum\limits^{n-1}\limits_{m_1=1}
                                            C_{n}^{m_1}\frac{m_1}{n}(u_{k-1}^{N}(t)-u_{k}^{N}(t))^{m_1}(v_{k-1}^{N}(t)-v_{k}^{N}(t))^{n-m_1}\\
                                 \nonumber&&\times\displaystyle\sum\limits^{d-n}\limits_{m=0}C^{m}_{d-n}(u_{k}^{N})^m(v_{k}^{N})^{d-m-n}\\
                                 \nonumber&=&\displaystyle(u_{k-1}^{N}(t)-u_{k}^{N}(t))\displaystyle\sum\limits_{n=2}\limits^{d}C^{n}_{d}
                                            \sum\limits^{n-1}\limits_{m_1=1}C_{n}^{m_1}\frac{m_1}{n}(u_{k-1}^{N}(t)-u_{k}^{N}(t))^{m_1-1}\\
                                          &\times&(v_{k-1}^{N}(t)-v_{k}^{N}(t))^{n-m_1}(u_{k}^{N}(t)+v_{k}^{N}(t))^{d-n}
                            \end{eqnarray}
                 \end{description}
              Thus, the probability that a new arriving task joins a working VM with exact $k-1 (k=2,3,\ldots,M-1)$ tasks, is given by
              \begin{equation*}
                (u_{k-1}^{N}(t)-u_{k}^{N}(t))W_{k},
              \end{equation*}
              where
              \begin{equation}\label{WkleqM-1}
                W_{k}=\sum\limits_{n=1}\limits^{d}C^{n}_{d}(u_{k}^{N}(t)+v_{k}^{N}(t))^{d-n}
                (u_{k-1}^{N}(t)-u_{k}^{N}(t)+v_{k-1}^{N}(t)-v_{k}^{N}(t))^{n-1}.
              \end{equation}
    \end{description}
    The analysis of the other two parts, $k=M$ and $k\geq M$, can be carried out similarly. However, it is worth noting that there are no dormant VMs with at least $M$ tasks, that is $v_{k}^{N}=0, k\geq M$. The probability that a new arriving task joins a working VM with exact $k-1,(k=M,M+1,\ldots)$ tasks, is given by
              \begin{equation*}
                (u_{k-1}^{N}(t)-u_{k}^{N}(t))W_{k},
              \end{equation*}
              where
              \begin{equation}\label{wk}
                W_{k}=\left\{\begin{array}{ll}
                               \sum\limits_{n=1}\limits^{d}C^{n}_{d}(u_{M}^{N}(t))^{d-n}
                                (u_{M-1}^{N}(t)-u_{M}^{N}(t)+v_{M-1}^{N}(t))^{n-1}, & k=M, \\
                                \sum\limits_{n=1}\limits^{d}C^{n}_{d}(u_{k-1}^{N}(t)-u_{k}^{N}(t))^{n-1}
                                (u_{k}^{N})^{d-n},& k=M+1,M+2,\cdots.
                             \end{array}
                \right.
              \end{equation}
  \item[Arrival process: joining a dormant VM] In order to avoid repeatability in the analysis, we give the following results directly. The probability that a new arriving task joins a dormant VM with exact $k-1,(k=1,2,3,\ldots,M-1)$ tasks, is given by
      \begin{equation*}
                (v_{k-1}^{N}(t)-v_{k}^{N}(t))V_{k},
              \end{equation*}
              where
       \begin{equation*}\label{vk}
                V_{k}=\left\{\begin{array}{ll}
                               \sum\limits_{n=1}\limits^{d}C_{d}^{n}(v_{0}^{N}(t)-v_{1}^{N}(t))^{n-1}(v_{1}^{N}(t)+
                                            u_{1}^{N}(t))^{d-n}, ~~~~~~~~~~~~~~~~~~k=1, \\
                               \sum\limits_{n=1}\limits^{d}C^{n}_{d}(u_{k}^{N}(t)+v_{k}^{N}(t))^{d-n}
                                (u_{k-1}^{N}(t)-u_{k}^{N}(t)+v_{k-1}^{N}(t)-v_{k}^{N}(t))^{n-1},& \\
                                ~~~~~~~~~~~~~~~~~~~~~~~~~~~~~~~~~~~~~~~~~~~~~~~~~~~~~~~~~~~~k=2,3,\ldots,M-1.
                             \end{array}
                \right.
       \end{equation*}
       The probability that a new arriving task joins a working VM with exact $M-1$ tasks is given by
       \begin{equation*}
         v_{M-1}V_{M}= v_{M-1}\sum\limits_{n=1}\limits^{d}C^{n}_{d}(u_{M}^{N}(t))^{d-n}(u_{M-1}^{N}(t)-u_{M}^{N}(t)+v_{M-1}^{N}(t))^{n-1}.
       \end{equation*}
       Noting that $W_{k}=V_{k}$, $k=2,3,\cdots,M$, in the rest paper,  we denote$W_1=V1$, and  all $V_k$ are expressed by $W_k$.

  \item[Service process] The service rate that a working VM with exact $k$ tasks completes a task is given by
     \begin{equation}
        N\mu(u_{k}^{N}(t)-u_{k+1}^{N}(t)).
     \end{equation}
\end{description}
Substituting the state transition and above probabilities into equation \eqref{Ant}, \eqref{An} can be easily derived.
\end{proof}

\section{Proofs to Propositions~\ref{p2.1} and \ref{pd}}
\label{appendix C}

\begin{proof}[Proof to Proposition~\ref{p2.1}.]
 Let \begin{eqnarray*}
                       x&=& (u_1,v_1;u_2,v_2;u_{3},v_{3};\ldots), \\
                     F(x)&=& (F_1(x),F_2(x),\cdots,F_{M-1}(x),F_M(x),\ldots),
                    \end{eqnarray*}
where,

\begin{equation}
  F_k(x)=\left\{\begin{array}{ll}
                   \left(\mu(u_{2}-u_{1})+\lambda v_{M-1}V_{M},  \lambda (v_{0}-v_{1})V_{k}-\lambda v_{M-1}V_{M}\right),~~~k=1,\\ [4mm]
                  \left( \lambda (u_{k-1}-u_{k})W_{k}+\mu(u_{k+1}-u_{k})+\lambda v_{M-1}V_{M}, \lambda (v_{k-1}-v_{k})V_{k}-\lambda v_{M-1}V_{M}\right),  \\[4mm]
                  ~~~~~~~~~~~~~~~~~~~~~~~~~~~~~~~~~~~~~~~~~~~~~~~~~~~~~~~~~~~~~~~~~~~~~~k=2,3,\ldots,M-1, \\[4mm]
                  (\lambda (u_{M-1}-u_{k})W_{k}+\mu(u_{M+1}-u_{M})+\lambda v_{M-1}V_{M},0)~~~~k=M,\\[4mm]
                  (\lambda (u_{k-1}-u_{k})W_{k}+\mu(u_{k+1}-u_{k}),0)~~~~~~~~~~~~~~~~~~~~~~~ ~k=M+1,M+2,\ldots.
                \end{array}
  \right.
\end{equation}
Note that $v_{k}=0$ for $k=M,M+1,\ldots$.

Then, the matrix of partial derivatives of the function $F(x)$ is given by
\begin{equation*}
  DF(x)=\left(
          \begin{array}{ccccc}
            A_1(x) & B_2(x) &        &      &  \\
            C_1(x) & A_2(x) & B_3(x) &       &   \\
                   & C_2(x)& A_3(x) & B_4(x) &   \\
                   &       &  \ddots&\ddots&\ddots\\
          \end{array}
        \right),
\end{equation*}
where
$A_{k},B_{k}$ and $C_{k}$, $k\geq1$ are $2\times 2$ matrices.
Sice $W_{k}=V_{k}$ $(k=2,3,\cdots,M)$, we have
\begin{equation*}
 A_{1}=\left(
       \begin{array}{cc}
       -\mu & \lambda(v_0-v_1)\frac{\partial W_1}{\partial u_1} \\[4mm]
       0 & -\lambda W_1+\lambda(v_0-v_1)\frac{\partial W_1}{\partial v_1} \\[4mm]
      \end{array}
     \right),
\end{equation*}
for $k=2,3,\ldots,M-2$,
\begin{equation*}
 A_{k}=\left(
       \begin{array}{cc}
       -\lambda W_{k}+\lambda(u_{k-1}-u_{k})\displaystyle\frac{\partial W_k}{\partial u_k}-\mu & \lambda(v_{k-1}-v_{k})\displaystyle\frac{\partial W_k}{\partial u_k} \\[4mm]
       \lambda(u_{k-1}-u_{k})\displaystyle\frac{\partial W_k}{\partial v_k} & -\lambda W_k+\lambda(v_{k-1}-v_{k})\displaystyle\frac{\partial W_k}{\partial v_k} \\[4mm]
      \end{array}
     \right),
\end{equation*}
and
\begin{equation*}
 A_{M-1}=\left(
      \footnotesize\displaystyle \begin{array}{cc}
       -\lambda W_{M-1}+\lambda(u_{M-2}-u_{M-1}+v_{M-1})\displaystyle\frac{\partial W_{M-1}}{\partial u_{M-1}}-\mu & \lambda(v_{M-2}-2v_{M-1})\displaystyle\frac{\partial W_{M-1}}{\partial u_{M-1}} \\[4mm]
       \lambda(u_{M-2}-u_{M-1}+v_{M-1})\displaystyle\frac{\partial W_{M-1}}{\partial v_{M-1}} &-\lambda W_{M-1}+\lambda(v_{M-2}-2v_{M-1})\displaystyle\frac{\partial W_{M-1}}{\partial v_{M-1}} \\[4mm]
      \end{array}
     \right),
\end{equation*}
\begin{equation*}
  A_{k}=\left(
          \begin{array}{cc}
           -\lambda W_{k}+\lambda(u_{k-1}-u_{k})\displaystyle\frac{\partial W_{k}}{\partial u_{k}}-\mu, & 0 \\[4mm]
            0 & 0 \\
          \end{array}
        \right),
   ~~~k=M,M+1,\cdots.
\end{equation*}
For $k=2,3,\ldots,M-1$,
\begin{equation*}
 B_{k}=\left(
       \begin{array}{cc}
       \lambda W_{k}+\lambda(u_{k-1}-u_{k})\displaystyle\frac{\partial W_k}{\partial u_{k-1}} & \lambda(v_{k-1}-v_{k})\displaystyle\frac{\partial W_k}{\partial u_{k-1}} \\[4mm]
       \lambda(u_{k-1}-u_{k})\displaystyle\frac{\partial W_k}{\partial v_{k-1}} & \lambda W_k+\lambda(v_{k-1}-v_{k})\displaystyle\frac{\partial W_k}{\partial v_{k-1}} \\[4mm]
      \end{array}
     \right),
\end{equation*}
and for $k= M,M+1,\ldots$
\begin{equation*}
 B_{M}=\left(
         \begin{array}{cc}
           \lambda W_{M}+\lambda(u_{M-1}-u_{M}+v_{M-1})\displaystyle\frac{\partial W_{M}}{\partial u_{M-1}} & 0 \\[4mm]
           \lambda(u_{M-1}-u_{M}+v_{M-1})\displaystyle\frac{\partial W_{M}}{\partial v_{M-1}}+\lambda W_{M})&0 \\[4mm]
         \end{array}
       \right),
\end{equation*}
\begin{equation*}
  B_{k}=\left(
          \begin{array}{cc}
            \lambda W_{k}+\lambda(u_{k-1}-u_{k})\displaystyle\frac{\partial W_{k}}{\partial u_{k-1}} & 0 \\[4mm]
            0 & 0 \\
          \end{array}
        \right),~~~~~k=M+1,M+2,\cdots.
\end{equation*}
Meanwhile, for $k=1,2,\ldots,M-3$ and $k=M,M+1,\ldots$
\begin{equation*}
  C_{k}=\left(\begin{array}{cc}
          \mu & 0 \\
          0 & 0
        \end{array}\right),
\end{equation*}
and
\begin{equation*}
  C_{M-2}=\left(\begin{array}{cc}
          \mu+\lambda v_{M-1}\displaystyle\frac{\partial W_{M}}{\partial u_{M-1}}& -\lambda v_{M-1}\displaystyle\frac{\partial W_{M}}{\partial u_{M-1}} \\[4mm]
          \lambda W_{M}+\lambda v_{M-1}\displaystyle\frac{\partial W_{M}}{\partial v_{M-1}} & -\left(\lambda W_{M}+\lambda v_{M-1}\displaystyle\frac{\partial W_{M}}{\partial v_{M-1}}\right),\\
        \end{array}\right),
\end{equation*}
\begin{equation*}
  C_{M-1}=\left(
            \begin{array}{cc}
               \mu+\lambda v_{M-1}\displaystyle\frac{\partial W_{M}}{\partial u_{M}} & -\lambda v_{M-1}\displaystyle\frac{\partial W_{M}}{\partial u_{M}} \\[4mm]
              0 & 0 \\
            \end{array}
          \right),
\end{equation*}

Overall, we have
\begin{eqnarray*}
 \| DF(x)\|&=&\max\left\{\| \textbf{e}(A_1(x)+C_1(x))\|,\sup\limits_{k\geq 2}\| \textbf{e}(C_{k}(x)+A_{k}(x)+B_{k}(x))\|\right\}\\
  &\leq& \sum\limits_{n=1}^{d}C_d^n(d-n)2^{d-2},
\end{eqnarray*}
where $\textbf{e}=(1,1)$.
Then, applying Lemma 5 in \cite{lql}, we have
\begin{eqnarray*}
 \|F(\textbf{u},\textbf{v})-F(\textbf{u}',\textbf{v})'\|&\leq&\sup\limits_{0\leq t\leq1}\| DF((\textbf{u},\textbf{v})+t((\textbf{u},\textbf{v})-(\textbf{u}',\textbf{v}')))\|\|(\textbf{u},\textbf{v})-(\textbf{u}',\textbf{v}'))\|\\[4mm]
&\leq& \sum\limits_{n=1}^{d}C_d^n(d-n)2^{d-2}\|(\textbf{u},\textbf{v})-(\textbf{u}',\textbf{v}'))\|.
\end{eqnarray*}
Therefore, the function $F(u,v)$ is Lipschitzian continuous and together with Picard approximation, it is easy to show that the deterministic process \eqref{d1}--\eqref{d2} has a unique solution.
\end{proof}

In Proposition~\ref{pd}, we denote that $A=(2K_1+K_2+K_3+K_4)\lambda+2\mu$, $A'=(2K_1+K_2+K_3+K_4)\lambda$, $B=(2K_2+2K_5+K_4+K_6)\lambda$, and $K_1$ to $K_6$ are given in (\ref{eqn:K1})--(\ref{eqn:K6}), respectively.

\begin{proof}[Proof of Proposition~\ref{pd}.]

 The proof follows the same line of the arguments as the proof to Lemma~3.2 in \cite{JB}.  Since the similarity in derivations, we only compute $\displaystyle \frac{\partial u_{k}}{\partial c_{uk}}$, which is marked as $u'_{k}$ for convenience.
\begin{eqnarray}
 \nonumber \frac{\partial u'_{1}}{\partial t} &=& \mu(u'_{2}(t)-u'_{1}(t))+\lambda v'_{M-1}(t)V_{M}+\lambda v_{M-1}(t)V'_{M}, \\
 \nonumber \frac{\partial u'_{k}}{\partial t} &=& \left\{\begin{array}{l}
                                                  \lambda (u'_{k-1}(t)-u'_{k}(t))W_{k}+\lambda (u_{k-1}(t)-u_{k}(t))W'_{k}+\mu(u'_{k+1}(t)-u'_{k}(t))\\[4mm]
                                                  ~~+\lambda v'_{M-1}(t)V_{M}+\lambda v_{M-1}(t)V'_{M},\\[4mm]
~~~~~~~~~~~~~~~~~~~~~~~~~~~~~~~~~~~~~~~~~~~~~~~~~~~~~~~~~~~~~~ k=2,3,\cdots,M-1,\\[4mm]
                                                  \lambda (u'_{k-1}(t)-u'_{k}(t))W_{k}+\lambda (u_{k-1}(t)-u_{k}(t))W'_{k}+\mu(u'_{k+1}(t)-u'_{k}(t)),\\[4mm]
 ~~~~~~~~~~~~~~~~~~~~~~~~~~~~~~~~~~~~~~~~~~~~~~~~~~~~~~~~~~~~~~  k=M,M+1,\cdots,
                                              \end{array}
                                          \right.
\end{eqnarray}
Applying the Lemma 3.1 in \cite{JB}, we have $A=2(K_1+K_2+K_3+K_4)\lambda+2\mu, b_0=0$ and $c=1$.
Thus,
\begin{equation*}
  \displaystyle\sup\limits_{k,i\geq1}\mid \frac{\partial u_{k}}{\partial c_{ui}}\mid\leq \exp \left\{((2K_1+K_2+K_3+K_4)\lambda +2\mu)t\right\}
\end{equation*}
\end{proof}



\end{document}